\documentclass[12pt,a4paper]{amsart}
\usepackage{amsmath, amssymb, amsfonts,enumerate}
\usepackage[all]{xy}
\usepackage{amscd}
\usepackage{comment}
\usepackage{mathtools}
\newtheorem{theorem}{Theorem}[section]
\newtheorem{lemma}[theorem]{Lemma}
\newtheorem{corollary}[theorem]{Corollary}
\newtheorem{proposition}[theorem]{Proposition}

 \theoremstyle{definition}

 \newtheorem{example}[theorem]{Example}

\numberwithin{equation}{section}
\newcommand {\Z}{\mathbb{Z}} 
\newcommand {\R}{\mathbb{R}} 
\newcommand {\Q}{\mathbb{Q}} 
\newcommand {\C}{\mathbb{C}} 

\newcommand{\T}{\mathbb{T}}

\newcommand{\CC}{\mathcal{C}}

\DeclareMathOperator{\Fix}{Fix}

\DeclareMathOperator{\Ker}{Ker}

\begin{document}
\title[A Garden of Eden theorem]{A Garden of Eden theorem for principal algebraic actions}
\author{Tullio Ceccherini-Silberstein}
\address{Dipartimento di Ingegneria, Universit\`a del Sannio, C.so
Garibaldi 107, 82100 Benevento, Italy}
\email{tullio.cs@sbai.uniroma1.it}
\author{Michel Coornaert}
\address{Universit\'e de Strasbourg, CNRS, IRMA UMR 7501, F-67000 Strasbourg, France}
\email{michel.coornaert@math.unistra.fr}
\begin{abstract}
Let $\Gamma$ be a countable abelian group  
and $f \in \Z[\Gamma]$, where $\Z[\Gamma]$ denotes the integral group ring  of $\Gamma$. 
Consider the Pontryagin dual  $X_f$ of the cyclic $\Z[\Gamma]$-module 
 $\Z[\Gamma]/\Z[\Gamma] f$ and 
suppose that  the natural action of $\Gamma$ on $X_f$ is expansive and that $X_f$ is connected.
We prove that if $\tau \colon X_f \to X_f$ is a $\Gamma$-equivariant continuous map, then $\tau$ is surjective  if and only if the restriction of $\tau$ to each $\Gamma$-homoclinicity class is injective.
 This is an analogue of the classical Garden of Eden theorem of Moore and Myhill for cellular automata with finite alphabet over $\Gamma$.
\end{abstract}
\date{\today}
\maketitle

\section{Introduction}

Consider a dynamical system $(X, \alpha)$, consisting of a compact metrizable space $X$,
called the \emph{phase space}, 
equipped with a continuous  action  $\alpha$ 
of a countable group $\Gamma$.
Let  $d$ be a metric on $X$ that is compatible with the topology.
Two points $x, y \in X$ are said to be \emph{homoclinic} if
$\lim_{\gamma \to \infty} d(\gamma x, \gamma y) = 0$, i.e.,
for every $\varepsilon > 0$, there exists a finite set $F \subset \Gamma$ such that  
$d(\gamma x,\gamma y) <  \varepsilon$ 
for all $\gamma \in \Gamma \setminus F$.
Homoclinicity is an equivalence relation on $X$.
By compactness of $X$, this relation does not depend of the choice of the compatible metric $d$.
 A map with source set   $X$ is called \emph{pre-injective}  (with respect to $\alpha$) 
 if its restriction to each homoclinicity class is injective.
  \par
An \emph{endomorphism} of the dynamical system $(X,\alpha)$ is a continuous map 
 $\tau \colon X \to X$ that is $\Gamma$-equivariant
 (i.e., $\tau(\gamma x) = \gamma \tau(x)$ for all $\gamma \in \Gamma$ and $x \in X$).
 \par
The original Garden of Eden theorem  is a statement in symbolic dynamics that characterizes surjective endomorphisms of shift systems with finite alphabet.
To be more specific, let us fix a compact metrizable space  $A$, called the \emph{alphabet}.
Given a countable group $\Gamma$,
the \emph{shift} over the group $\Gamma$ with alphabet $A$ is the dynamical system 
$(A^\Gamma,\sigma)$, where
$A^\Gamma = \{x \colon \Gamma \to A\}$ is equipped with the product topology and
  $\sigma $ is the action defined by
  $\gamma x(\gamma') \coloneqq  x(\gamma^{-1} \gamma')$ for all $x \in A^\Gamma$ and $\gamma,\gamma' \in \Gamma$.
The \emph{Garden of Eden theorem}  states that, under the hypotheses that the group  $\Gamma$ is amenable and the alphabet $A$ is finite,    an endomorphism of $(A^\Gamma,\sigma)$ is surjective if and only if it is pre-injective.
It was first proved for  $\Gamma = \Z^d$ by Moore and Myhill in the early 1960s.   
Actually, the implication
surjective $\implies$ pre-injective  was first proved by Moore in~\cite{moore}
while the converse implication was established shortly after by
Myhill in~\cite{myhill}.
The Garden of Eden theorem was subsequently extended to
finitely generated groups of subexponential growth by Mach{\`{\i}} and Mignosi~\cite{machi-mignosi} 
and finally to  all countable amenable groups by
Mach{\`{\i}},  Scarabotti, and the first author in~\cite{ceccherini}.
 \par
Let us  say that the  dynamical system $(X,\alpha)$ has the \emph{Moore property} if every surjective endomorphism of $(X,\alpha)$ is pre-injective and that it has the 
\emph{Myhill property} if 
every pre-injective endomorphism of $(X,\alpha)$ is surjective.
We say that the dynamical system $(X,\alpha)$ has the \emph{Moore-Myhill property},
or that it satisfies the \emph{Garden of Eden theorem},
if it has both the Moore and the Myhill properties.
It turns out that both the Moore and the Myhill properties  for shifts with finite alphabet characterize amenable groups.
Indeed, if $\Gamma$ is  non-amenable,
on one hand,
Bartholdi~\cite{bartholdi-ejm-2010} proved the existence of a finite alphabet $A$ such that 
$(A^\Gamma,\sigma)$ does not have the Moore property, and, on the other, Bartholdi and Kielak~\cite{bartholdi:2016}
proved the existence of a finite alphabet $B$ such that $(B^\Gamma,\sigma)$ does not have the Myhill property.
 \par
The  dynamical system  $(X,\alpha)$ is called \emph{expansive} if there exists a constant 
$\delta  > 0$ such that, for every pair of distinct points $x,y \in X$, there exists an element
$\gamma   \in \Gamma$ such that
$d(\gamma  x,\gamma  y) \geq  \delta$.
Such a constant $\delta$ is  called an \emph{expansiveness constant} for $(X,\alpha,d)$.  
The fact that $(X,\alpha)$ is expansive or not does not depend on the choice of the metric $d$.
For instance, the shift system $(A^\Gamma,\sigma)$ is expansive for every countable group $\Gamma$ whenever the alphabet $A$ is finite.
\par
The goal of the present paper is to establish a version of the Garden of Eden theorem for expansive principal algebraic systems with connected phase space over countable abelian groups. 
By an \emph{algebraic dynamical system}, we mean a dynamical system of the form 
$(X,\alpha)$, where $X$ is a compact metrizable abelian group and $\alpha$ is an  action of a countable group $\Gamma$ 
 on $X$ by  continuous group automorphisms.
 By Pontryagin duality, algebraic dynamical systems with acting group $\Gamma$ are in one-to-one correspondence with countable left 
 $\Z[\Gamma]$-modules. 
 Here $\Z[\Gamma]$ denotes the integral group ring of $\Gamma$.
 This correspondence has been intensively studied in the last decades 
 and revealed fascinating connections  between commutative algebra, number theory, harmonic analysis, ergodic theory,   and dynamical systems (see in particular the monograph~\cite{schmidt-book} and the survey~\cite{lind-schmidt-survey-heisenberg}).
 \par 
Let $f \in \Z[\Gamma]$ and consider the cyclic left $\Z[\Gamma]$-module
 $M_f \coloneqq \Z[\Gamma]/  \Z[\Gamma] f$ obtained by quotienting the ring $\Z[\Gamma]$ by the principal left ideal generated by $f$.
 The algebraic dynamical system  associated by Pontryagin duality with $M_f$
 is denoted by $(X_f,\alpha_f)$ and is 
 called the \emph{principal algebraic dynamical system} associated with $f$.
 \par 
Our main result is the following.

\begin{theorem}
\label{t:main-result}
Let  $\Gamma$ be a countable abelian group (e.g.~$\Gamma = \Z^d$) 
and $f \in \Z[\Gamma]$. 
Suppose that  the principal algebraic dynamical system $(X_f,\alpha_f)$ associated with $f$ is expansive and that  $X_f$ is connected.
Then the  dynamical system $(X_f,\alpha_f)$ has the Moore-Myhill property.
\end{theorem}

There are two main ingredients in our proof of this theorem.
The first one is a rigidity result of Bhattacharya~\cite{bhattacharya} for expansive algebraic systems with connected phase space.
We use it to prove that every endomorphism of $(X_f,\alpha_f)$ is affine with linear part of the form 
$x \mapsto r x$ for some $r \in \Z[\Gamma]$.
The second one is a result of Lind and Schmidt~\cite{lind-schmidt} which asserts that
the homoclinic group of $(X_f,\alpha_f)$ (i.e., the subgroup of $X$ consisting of all points homoclinic to $0_X$), equipped with the discrete topology and the induced action of $\Gamma$,  is  conjugate to
the Pontryagin dual $\Z[\Gamma]/ \Z[\Gamma] f$ of $X_f$.
\par 
Our motivation for the present work originated from a sentence of Gromov
\cite[p.~195]{gromov-esav}
suggesting that the Garden of Eden theorem could be extended to dynamical systems with a suitable hyperbolic flavor other than shifts and subshifts.
A first step in that direction  was made in~\cite{csc-anosov-tori}, where we proved that all Anosov diffeomorphisms on tori generate $\Z$-actions  with  the Moore-Myhill property,
and another one in~\cite{csc-ijm-2015},
where we gave sufficient conditions for  expansive actions of countable amenable groups
to have the Myhill property. 
\par
The paper is organized as follows.
Section~\ref{sec:background} introduces notation and collects background material on algebraic dynamical systems. 
In Section~\ref{sec:affine}, 
we discuss  topological rigidity of expansive algebraic dynamical systems.
The proof of Theorem~\ref{t:main-result}  
is given in Section~\ref{sec:proof}.
In the last section, we collect some
final remarks and exhibit some examples showing that Theorem~\ref{t:main-result} becomes false if the expansiveness hypothesis is replaced by the weaker hypothesis that the system $(X_f,\alpha_f)$ is mixing.

\section{Background material and preliminaries}
\label{sec:background}

\subsection{Group actions}
Let $\Gamma$ be a countable group. 
We use multiplicative notation for the group operation in $\Gamma$ and denote by $1_\Gamma$ its identity element.
\par 
An \emph{action} of  $\Gamma$ on a set $X$
is a map $\alpha \colon \Gamma \times X \to X$ such that
$\alpha(1_\Gamma,x) = x$ and
$\alpha(\gamma_1,\alpha(\gamma_2,x)) = \alpha(\gamma_1 \gamma_2,x)$ for all $\gamma_1,\gamma_2 \in \Gamma$ and $x \in X$. 
In the sequel, to simplify, we shall write
$\gamma x$ instead of $\alpha(\gamma, x)$,  if there is no risk of confusion.
\par
If $\alpha$ is an action of $\Gamma$ on a set $X$, we denote by $\Fix(X,\alpha)$ the set of points of $X$ that are \emph{fixed} by $\alpha$, i.e., the set of points $x \in X$ such that $\gamma x = x$ for all 
$\gamma \in \Gamma$.  
\par
If $\Gamma$ acts on two sets $X$ and $Y$, a map
$\tau \colon X \to Y$ is said to be $\Gamma$-\emph{equivariant} if one has
$\tau(\gamma x) = \gamma \tau(x)$ for all $\gamma \in \Gamma$ and $x \in X$.

\subsection{Convolution}
Let $\Gamma$ be a countable group.
We denote by $\ell^\infty(\Gamma)$  the vector space  consisting of all formal series
$$
f = \sum_{\gamma \in \Gamma} f_\gamma \gamma,
$$
with coefficients $f_\gamma \in \R$ for all $\gamma \in \Gamma$ and 
$$
\Vert f \Vert_\infty \coloneqq \sup_{\gamma \in \Gamma} |f_\gamma| < \infty.
$$
We denote by $\ell^1(\Gamma)$ the vector subspace of $\ell^\infty(\Gamma)$  
consisting of all $f \in \ell^\infty(\Gamma)$ such that
$$
\Vert f \Vert_1 \coloneqq \sum_{\gamma \in \Gamma} |f_\gamma| < \infty.
$$
The \emph{convolution product} of $f \in \ell^\infty(\Gamma)$ and $g \in \ell^1(\Gamma)$ is the element $f g \in \ell^\infty(\Gamma)$ defined by 
$$
(f g)_\gamma \coloneqq
\sum_{\substack{\gamma_1, \gamma_2 \in \Gamma:\\ \gamma_1\gamma_2 = \gamma}}  f_{\gamma_1} g_{\gamma_2}. 
 $$
 We have the
the associativity rule
$(f g) h = f (g h)$ for all $f \in \ell^\infty(\Gamma)$ and $g,h \in \ell^1(\Gamma)$.
\par  
The vector space $\ell^1(\Gamma)$ is a Banach *-algebra for the norm $\vert \cdot \vert_1$,
the convolution product, and the involution $f \mapsto f^*$ defined by
  \[
 (f^*)_\gamma \coloneqq  f_{\gamma^{-1}}
 \]
 for all $f\in \ell^1(\Gamma)$ and $\gamma \in \Gamma$.
 \par
  The \emph{integral group ring} $\Z[\Gamma]$ is the
subring of $\ell^1(\Gamma)$ consisting of all $f \in \ell^1(\Gamma)$ 
such that  $f_\gamma \in \Z$ for all $\gamma \in \Gamma$ and $f_\gamma = 0$ for all but finitely 
many $\gamma \in \Gamma$.
\par
Observe that the convolution product  extends the group operation on 
 $\Gamma \subset \Z[\Gamma]$.
\par
Note also that, as a $\Z$-module, $\Z[\Gamma]$ is free with base $\Gamma$. 
\par
If we take $\Gamma = \Z^d$, then $\Z[\Gamma]$ is the Laurent polynomial ring 
$R_d \coloneqq \Z[u_1^{\pm 1}, \dots, u_d^{\pm 1}]$ on $d$ commuting indeterminates 
$u_1,\dots, u_d$.

\subsection{Pontryagin duality}
\label{subsec:pontryagin}
 Let us briefly review some basic facts and results regarding Pontryagin duality.
For more details and complete proofs,
the reader is refered to~\cite{morris}. 
\par
Let $X$ be an LCA group, i.e., a locally compact, Hausdorff, abelian topological group.
A continuous group morphism from $X$ into the circle $\T \coloneqq \R/\Z$
is called a \emph{character} of $X$.
The set $\widehat{X}$ of all characters of $X$,
equipped with pointwise multiplication and the topology of uniform convergence on compact sets,
is also an LCA  group,
  called  the \emph{character group} or  \emph{Pontryagin dual}  of $X$.
  \par
The natural map $\langle \cdot, \cdot  \rangle \colon \widehat{X} \times X \to \T$,
given by $\langle \chi, x \rangle = \chi(x)$ for all $x \in X$ and $\chi \in \widehat{X}$ is bilinear and non-degenerate. 
Moreover, 
the evaluation map $\iota \colon X \to \widehat{\widehat{X}}$,
defined by $\iota(x)(\chi) \coloneqq \langle\chi,x\rangle$, 
is a topological group    isomorphism  from $X$ onto its bidual 
$\widehat{\widehat{ X}}$. 
This canonical isomorphism is used to identify $X$ with $\widehat{\widehat{ X}}$.
\par 
The space $X$ is
compact (resp.~discrete, resp.~metrizable, resp.~$\sigma$-compact)
if and only if $\widehat{X}$ is
discrete (resp.~compact, resp.~$\sigma$-compact, resp.~metrizable).
In particular, $X$ is compact and metrizable if and only if $\widehat{X}$ is discrete and countable.
When $X$ is compact, $X$ is connected if and only if $\widehat{X}$ is a \emph{torsion-free} group  (i.e., a group with  no non-trivial elements of finite order).
\par
If $X$ is an LCA group and $Y$ a closed subgroup of $X$, then $X/Y$ is an LCA group
whose Pontryagin dual is canonically isomorphic, as a topological group,
to the closed subgroup $Y^\perp$ of $\widehat{X}$ defined by
\[
Y^\perp \coloneqq \{\chi \in \widehat{X} : \langle \chi, y \rangle = 0 \text{ for all } y \in Y\}.
\]
\par  
Let $X, Y$ be LCA groups and $\varphi \colon X \to Y$ a continuous group morphism.
The map $\widehat{\varphi} \colon \widehat{Y} \to \widehat{X}$, defined by 
$\widehat{\varphi}(\chi) \coloneqq  \chi \circ \varphi$ for all $\chi \in \widehat{Y}$ is a continuous group morphism, called the \emph{dual} of $\varphi$.
If we identify $X$ and $Y$ with their respective biduals,
then $\widehat{\widehat{\varphi}} =\varphi $.
If $\varphi$ is surjective, then $\widehat{\varphi}$ is injective.
It may happen that $\varphi$ is injective while $\widehat{\varphi}$ is not surjective. 
However, if $\varphi$ is both injective and open, then $\widehat{\varphi}$ is surjective.
As a consequence, if $X$ and $Y$ are either both compact or both discrete, then  $\varphi$ is injective (resp.~surjective) if and only if $\widehat{\varphi}$ is surjective (resp.~injective)  
\cite[Proposition~30]{morris}.
\par
Let $X$ be an LCA group 
and suppose that there is a countable group $\Gamma$ acting continuously on $X$ by  group automorphisms. By linearity, this action induces a left $\Z[\Gamma]$-module structure on $X$.
There is a   dual action of $\Gamma$ on $\widehat{X}$ by continuous group automorphisms, defined by 
\[
\gamma  \chi(x) \coloneqq \chi(\gamma^{-1}  x) \quad \text{for all } \gamma \in \Gamma, x \in X, \text{ and } \chi \in \widehat{X}.
\]
Therefore  there is also a   left $\Z[\Gamma]$-module structure on $\widehat{X}$. 
Note that the canonical topological group isomorphism $\iota \colon X \to \widehat{\widehat{X}}$ is   
$\Gamma$-equivariant and hence a left $\Z[\Gamma]$-module isomorphism.

\subsection{Algebraic dynamical systems}
An \emph{algebraic dynamical system}
is a pair $(X,\alpha)$, where $X$ is a compact metrizable abelian topological group
and $\alpha$ is an action of a countable group $\Gamma$ on $X$ by continuous group automorphisms.
\par
As an example, if $A$ is a compact metrizable abelian topological group (e.g.~$A = \T$) and 
$\Gamma$ a countable group,
then the system $(A^\Gamma,\sigma)$, where
$A^\Gamma = \{x \colon \Gamma \to A\}$ is equipped with the product topology,  and $\sigma$ is the \emph{shift action}, defined by
\[
(\sigma(\gamma,x))(\gamma') \coloneqq x(\gamma^{-1} \gamma') \quad \text{for all  } \gamma, \gamma' \in \Gamma \text{ and }x \in A^\Gamma,
\]
is an algebraic dynamical system.
\par
Let $(X,\alpha)$ be an algebraic dynamical system with acting group $\Gamma$.
As $X$ is compact and metrizable, its Pontryagin dual $\widehat{X}$ is a discrete countable abelian group.
We have seen at the end of the previous subsection that there is a left $\Z[\Gamma]$-module structure on $\widehat{X}$ induced by the action of $\Gamma$ on $X$.
Conversely, if $M$ is a countable left $\Z[\Gamma]$-module and we equip $M$ with its discrete topology, then its Pontryagin dual $\widehat{M}$ is a compact metrizable abelian group
and there is, by duality, an action $\alpha_M$ of $\Gamma$ on $\widehat{M}$ by continuous group automorphisms,
so that
$(\widehat{M},\alpha_M)$ is an algebraic dynamical system.
In this way, algebraic dynamical systems with acting group $\Gamma$ are in one-to-one correspondence with countable left $\Z[\Gamma]$-modules.

\subsection{Principal algebraic dynamical systems}
Let $\Gamma$ be a countable group
and  $f = \sum_{\gamma \in \Gamma} f_\gamma \gamma \in \Z[\Gamma]$.
Consider the left $\Z[\Gamma]$-module $M_f \coloneqq  \Z[\Gamma]/ \Z[\Gamma] f$, where   
$ \Z[\Gamma] f$ is the principal left ideal of $\Z[\Gamma]$ generated by $f$.
To simplify notation, let us write $X_f$ instead of $X_{M_f}$ and $\alpha_f$ instead 
of $\alpha_{M_f}$. 
The  algebraic dynamical system $(X_f,\alpha_f)$ is called
the \emph{principal algebraic dynamical system} associated with $f$.
\par
One  can regard $(X_f,\alpha_f)$ as a \emph{subshift} of 
$(\T^\Gamma,\sigma)$, i.e., 
as a closed  subset of $\T^\Gamma$
that is invariant under the  shift action $\sigma$ of $\Gamma$ on $\T^\Gamma$, in the following way. The Pontryagin dual of $\T^\Gamma$ is $\Z[\Gamma]$ with pairing
$\langle \cdot,\cdot \rangle \colon \Z[\Gamma] \times \T^\Gamma \to \T$ given by
\[
\langle g, x \rangle = \sum_{\gamma \in \Gamma} g_{\eta} x(\eta) \quad 
\text{for all } g \in \Z[\Gamma], x \in \T^\Gamma.
\]
Therefore
\begin{align*}
X_f
&= \widehat{\Z[\Gamma]/\Z[\Gamma] f} \\
&= (\Z[\Gamma] f)^\perp \\
&= \{ x \in \T^\Gamma : \langle g,x\rangle = 0 \text{ for all } g \in \Z[\Gamma] f \} \\
&= \{ x \in \T^\Gamma : \langle \gamma f,x\rangle = 0 \text{ for all } \gamma \in \Gamma  \}, 
\end{align*}
that is,
\begin{equation}
\label{e:X-f-subshift}
X_f= \{ x \in \T^\Gamma : \sum_{\eta \in \Gamma} f_{\eta} x(\gamma \eta)  = 0 \text{ for all } \gamma \in \Gamma \}, 
\end{equation}
with the action $\alpha_f$ of $\Gamma$ on $X_f \subset \T^\Gamma$ obtained  by restricting to $X_f$  the  shift action $\sigma$.
\par
Consider the surjective map $\pi \colon \ell^\infty(\Gamma) \to \T^\Gamma$
defined by $\pi(g)(\gamma) = g_\gamma \mod 1$ for all $g \in \ell^\infty(\Gamma)$ and 
$\gamma \in \Gamma$.
Denote by $\ell^ \infty(\Gamma,\Z)$ the set consisting of all  $g \in \ell^\infty(\Gamma)$ such that $g_\gamma \in \Z$ for all $\gamma \in \Gamma$.

\begin{proposition}
\label{p:X-f-carac-lift}
Let $x \in \T^\Gamma$ and $g \in \ell^\infty(\Gamma)$ such that $\pi(g) = x$.
With the above notation,  the following conditions are equivalent:
\begin{enumerate}[\rm (a)]
\item
$x \in X_f$;
\item
$g f^* \in \ell^\infty(\Gamma,\Z)$.
\end{enumerate}
\end{proposition}
  
\begin{proof}
By~\eqref{e:X-f-subshift}, we see that $x$ is in $X_f$ if and only if
\[
\sum_{\eta \in \Gamma} f_{\eta} g_{\gamma \eta} \in \Z
\]
for all $\gamma \in \Gamma$.
Now it suffices to observe that, by definition of the convolution product,
 \[
 (g f^*)_\gamma = 
 \sum_{\eta \in \Gamma} g_{\gamma \eta} f^*_{\eta^{-1}} =
 \sum_{\eta \in \Gamma} f_{\eta} g_{\gamma \eta}  
 \]
for all $\gamma \in \Gamma$.
\end{proof}

The following result is due to Deninger and Schmidt~\cite[Theorem~3.2]{deninger-schmidt}
(see also \cite[Theorem~5.1]{lind-schmidt-survey-heisenberg}). 

\begin{theorem}
\label{t:pads-expansive}
Let $\Gamma$ be a countable group and $f \in \Z[\Gamma]$.
Then the following conditions are equivalent:
\begin{enumerate}[\rm (a)]
\item
the dynamical system $(X_f,\alpha_f)$ is expansive;
\item
$f$ is invertible in $\ell^1(\Gamma)$.
\end{enumerate}
\end{theorem}

As observed in~\cite{lind-schmidt-survey-heisenberg}, 
  if  $f$ is \emph{lopsided}, i.e., there exists an element  
$\gamma_0 \in \Gamma$ such that
$|f_{\gamma_0}| > \sum_{\gamma \not= \gamma_0} |f_\gamma|$, 
then $f$ is invertible in $\ell^1(\Gamma)$.
On the other hand, there are $f \in \Z[\Gamma]$ invertible in $\ell^1(\Gamma)$ that are not lopsided.
For instance, if we take $\Gamma = \Z$, then the polynomial
$u^2 - u - 1 \in \Z[\Gamma] = \Z[u,u^{-1}]$
 is not lopsided although it is invertible in $\ell^1(\Gamma)$ 
(the associated principal algebraic dynamical system  is 
conjugate to the $\Z$-system generated by Arnold's cat map $(x_1,x_2) \mapsto (x_2,x_1 + x_2)$ on   the $2$-dimensional torus $\T^2$, see e.g.~\cite[Example~2.18.(2)]{schmidt-book}).
\par
A non-zero element $f \in \Z[\Gamma]$ is called \emph{primitive} if there is no integer $n \geq 2$ that divides all coefficients of $f$. Every nonzero element $f \in \Z[\Gamma]$ can be uniquely written in the form
$f = m f_0$ with $m$ a positive integer and $f_0$ primitive. The integer $m$ is called the \emph{content} of $f$.
In the case $\Gamma = \Z^d$, we have the following criterion for the connectedness of $X_f$.

\begin{proposition}
\label{p:Xf-connectedness}
Let $\Gamma = \Z^d$. Let $f \in \Z[\Gamma]$
with $f \not= 0$.
Then the following conditions are equivalent:
\begin{enumerate}[\rm (a)]
\item
$X_f$ is connected;
\item
$f$ is primitive.
\end{enumerate} 
\end{proposition}

\begin{proof}
We know that $X_f$ is connected if and only if its Pontryagin dual
$M_f = \Z[\Gamma]/\Z[\Gamma] f$ is torsion-free as a $\Z$-module.
Let $m$ denote the content of $f$.
By Pontryagin duality, $(\Z/m\Z)^\Gamma$ is a quotient of $X_f$  
since $\Z[\Gamma] f$ is a subgroup of $\Z[\Gamma] m$.
Therefore,
if $X_f$ is connected then $(\Z/m\Z)^\Gamma$ must be also connected.
As this last condition implies $m = 1$, i.e.,   $f$  primitive, this shows (a) $\Rightarrow$  (b).
Conversely, suppose that $M_f$ contains an  element $q \not= 0$ 
of finite order $n \geq 2$.
If $g \in \Z[\Gamma]$ is a representative of $q$, then $n g = h f$ for some $h \in \Z[\Gamma]$.
Using the fact that $\Z[\Gamma]$ is a unique factorization domain, we deduce that $n$ divides the content of $f$, so that  $f$ is not primitive.
This shows (b) $\Rightarrow$ (a).  
\end{proof}

 \subsection{The homoclinic group}
Let $(X,\alpha)$ be an algebraic dynamical system with acting group $\Gamma$.
The set of points in $X$ that are homoclinic to $0_X$ with respect to $\alpha$ is 
a $\Z[\Gamma]$-submodule  $\Delta(X,\alpha) \subset X$,  which is called the \emph{homoclinic group} of $(X,\alpha)$  (cf.~\cite{lind-schmidt}, \cite{lind-schmidt-survey-heisenberg}).
Note that $x \in \Delta(X,\alpha)$ if and only if
$\lim_{\gamma \to \infty} \gamma x = 0_X$.
We can choose a compatible metric $d$ on $X$ that is translation-invariant so that
$$
d(\gamma x, \gamma y) = d(\gamma x - \gamma y, 0_X) = d(\gamma(x - y),0_X)
$$
for all $x,y \in X$ and $\gamma \in \Gamma$.
We deduce that $x$ and $y$ are homoclinic if and only if $x - y \in \Delta(X,\alpha)$.
\par
Denote by $\CC_0(\Gamma)$ the vector space consisting of all $g \in \ell^\infty(\Gamma)$ that vanish at infinity, i.e.,  such that
$\lim_{\gamma \to \infty} g_\gamma = 0$.

\begin{lemma}
\label{l:char-homoclinic-shift}
Let $\Gamma$ be a countable group and let $x \in \T^\Gamma$. 
The following conditions are equivalent.
\begin{enumerate}[{\rm (a)}]
\item 
$x \in \Delta(\T^\Gamma,\sigma)$; 
\item 
$\lim_{\gamma \to \infty} x(\gamma) = 0_\T$;
\item
there exists $g \in \CC_0(\Gamma)$ such that $x = \pi(g)$.
\end{enumerate}
\end{lemma}

\begin{proof}
Suppose (a). Then $\lim_{\gamma \to \infty} \gamma^{-1} x = 0_{\T^\Gamma}$.
As $x(\gamma) = \gamma^{-1}x(1_\Gamma)$, this implies
\[
\lim_{\gamma \to \infty} x(\gamma) = \lim_{\gamma \to \infty} \gamma^{-1}x(1_\Gamma) = 0_\T 
\]
for all $\gamma \in \Gamma$.
This shows (a) $\implies$ (b).
\par
Conversely, suppose (b). 
Let $W$ be a neighborhood of $0_{\T^\Gamma}$ in $\T^\Gamma$ 
and let us show that there exists a finite
subset $\Omega \subset \Gamma$ such that
\begin{equation}
\label{e:omega}
\gamma x \in W \mbox{ for all } \gamma \in \Gamma \setminus \Omega.
\end{equation}
By definition of the product topology, we can find a neighborhood $V$ of $0_\T$ in $\T$
and a finite subset $\Omega_1 \subset \Gamma$
such that 
$W$ contains all $y \in \T^\Gamma$ that satisfy
\begin{equation*}
\label{e:W-V-Omega1}
y(\omega_1) \in V \mbox{ for all } \omega_1 \in \Omega_1.
\end{equation*}
On the other hand, since $\lim_{\gamma \to \infty} x(\gamma) = 0_\T$, we can find a finite subset 
$\Omega_2 \subset \Gamma$
such that
\begin{equation}
\label{e:V}
x(\gamma) \in V \mbox{ for all } \gamma \in \Gamma  \setminus \Omega_2.
\end{equation}
Take  $\Omega \coloneqq \Omega_1 \Omega_2^{-1} \subset \Gamma$
and suppose that $\gamma \in \Gamma \setminus  \Omega$.
Then for every $\omega_1 \in \Omega_1$, we have that $\gamma^{-1} \omega_1 \in \Gamma \setminus \Omega_2$ and hence 
\[
\gamma x(\omega_1) = x(\gamma^{-1} \omega_1)  \in V
\]
by~\eqref{e:V}. 
This implies that  $\gamma x \in  W$. 
Thus \eqref{e:omega} is satisfied.
This proves (b) $\implies$ (a).
\par
The fact that (c) implies (b) is an immediate consequence of the continuity of the quotient map 
$\R \to \T = \R/\Z$.
Conversely, if we assume  (b), then the unique $g \in \ell^\infty(\Gamma)$ such that 
$-1/2 \leq g_\gamma < 1/2$ and $x(\gamma) = g_\gamma \mod 1$ for all $\gamma \in \Gamma$ clearly satisfies (c).   
\end{proof}

The following result is due to Lind and Schmidt~\cite{lind-schmidt} 
(see also~\cite[Section 6]{lind-schmidt-survey-heisenberg}). 

\begin{theorem}[Lind and Schmidt]
\label{t:homoclinic-group-dual}
Let $\Gamma$ be a countable group  and $f \in \Z[\Gamma]$.
Suppose that the algebraic dynamical system $(X_f,\alpha_f)$ associated with $f$ is expansive.
Then, 
the homoclinic group $\Delta(X_f,\alpha_f)$
is isomorphic, as a left $\Z[\Gamma]$-module,  to $\Z[\Gamma]/\Z[\Gamma] f^*$.
\end{theorem}

\begin{proof}(see~\cite[Lemma~4.5]{lind-schmidt} and \cite{lind-schmidt-survey-heisenberg}).
By Theorem~\ref{t:pads-expansive}, the expansiveness of $(X_f,\alpha_f)$ implies that $f$, and hence $f^*$,  are invertible in 
$\ell^1(\Gamma)$.
Let $w \in \ell^1(\Gamma)$ denote the inverse of $f^*$  and
consider the element
$x_f^\Delta \in \T^\Gamma$ defined by
$x_f^\Delta \coloneqq \pi(w)$.
As $w f^* = 1_\Gamma \in \ell^\infty(\Gamma,\Z)$,
we deduce from Proposition~\ref{p:X-f-carac-lift} that $x_f^\Delta \in X_f$. 
On the other hand, we have that
$w \in \ell^1(\Gamma) \subset \CC_0(\Gamma)$, so that 
$x_f^\Delta \in \Delta(\T^\Gamma,\sigma) \cap X_f = \Delta(X_f,\alpha_f)$ by 
Lemma~\ref{l:char-homoclinic-shift}.
Consider now the left $\Z[\Gamma]$-module morphism  
$\Psi \colon \Z[\Gamma] \to  \Delta(X_f, \alpha_f)$ 
given by 
$\Psi(h) =h x_f^\Delta $.
We claim that $\Psi$ is surjective.
To see this, let $x \in \Delta(X_f,\alpha_f)$.
By Lemma~\ref{l:char-homoclinic-shift}, there exists $g \in \CC_0(\Gamma)$ such that
 $x = \pi(g)$.
Since $x \in X_f$, it follows from Proposition~\ref{p:X-f-carac-lift} that
 $h \coloneqq g f^* \in \ell^\infty(\Gamma,\Z)$.
As $h \in \CC_0(\Gamma)$, we deduce that $h  \in \Z[\Gamma]$
and $x = \Psi(h)$. This proves  our claim that $\Psi$ is surjective.
On the other hand,  $h \in \Z[\Gamma]$ is in the kernel of $\Psi$ if and only if $h w$ has integral coefficients.
As $h w \in \CC_0(\Gamma)$,
this is equivalent to $h w \in \Z[\Gamma]$
and hence to  $h \in \Z[\Gamma] f^*$. 
This shows that  $\Ker(\Psi) = \Z[\Gamma] f^*$.
Since $\Psi$ is surjective, it induces a
left $\Z[\Gamma]$-module isomorphism from $\Z[\Gamma]/\Z[\Gamma] f^*$ onto 
$\Delta(X_f,\alpha_f)$.
\end{proof}

\begin{corollary}[Lind and Schmidt]
\label{c:homoclinic-group-dual}
Let $\Gamma$ be a countable abelian group  and $f \in \Z[\Gamma]$.
Suppose that the algebraic dynamical system $(X_f,\alpha_f)$ associated with $f$ is expansive.
Then, 
the homoclinic group $\Delta(X_f,\alpha_f)$
is isomorphic, as a  $\Z[\Gamma]$-module,  to 
the Pontryagin dual $\widehat{X_f} = \Z[\Gamma]/\Z[\Gamma] f$.
\end{corollary}

\begin{proof}
Since $\Gamma$ is abelian,
the map $\gamma \mapsto \gamma^{-1}$ is an automorphism of $\Gamma$ and induces a 
$\Z[\Gamma]$-module isomorphism from
$\Z[\Gamma]/\Z[\Gamma] f$ onto $\Z[\Gamma]/\Z[\Gamma] f^*$.
\end{proof}
\section{Topological rigidity}
\label{sec:affine}

\subsection{Affine maps} 
 Let $X$ be a topological abelian group.
 A map $\tau \colon X \to X$ is called \emph{affine} if there is a continuous group morphism 
$\lambda \colon X \to X$ and an element $t \in X$ such that
$\tau(x) = \lambda(x) + t$ for all $x \in X$.
Note that  $\lambda$ and  $t$ are then uniquely determined by $\tau$
since they must satisfy $t = \lambda(0_X)$ and $\lambda(x) = \tau(x) - t$ for all $x \in X$. 
One says that 
$\lambda$ and $t$ are respectively the \emph{linear part}  and the \emph{translational part} of the affine map $\tau$.
\par
The following two obvious criteria will be useful in the sequel. 

\begin{proposition}
\label{p:pre-inj-affine}
Let $(X,\alpha)$ be an algebraic dynamical system and let $\tau \colon X \to X$ be an affine map
with linear part $\lambda \colon X \to X$.
Then the following conditions are equivalent:
\begin{enumerate}[\rm (a)]
\item
$\tau$ is pre-injective;
\item
$\lambda$ is pre-injective;
\item
$\Ker(\lambda) \cap \Delta(X,\alpha) = \{0_X\}$.\end{enumerate}
\end{proposition}

\begin{proposition}
\label{p:charact-equiv-affine}
Let $(X,\alpha)$ be an algebraic dynamical system and let $\tau \colon X \to X$ be an affine map
with linear part $\lambda \colon X \to X$ and translational part $t \in X$.
Then the following conditions are equivalent:
\begin{enumerate}[\rm (a)]
\item
$\tau$ is $\Gamma$-equivariant;
\item
$\lambda$ is $\Gamma$-equivariant and $t \in \Fix(X,\alpha)$.
\end{enumerate}
\end{proposition}

\subsection{Topological rigidity}
One says that the algebraic dynamical system $(X,\alpha)$ is \emph{topologically rigid} if every endomorphism $\tau \colon X \to X$ of $(X,\alpha)$ is affine. 
\par
The following result is due to Bhattacharya~\cite[Corollary~1]{bhattacharya}.

\begin{theorem}
\label{t:bhat-rigid}
Let $\Gamma$ be a countable group.
Let $(X,\alpha)$ be an algebraic dynamical system with acting group $\Gamma$.
Suppose that $(X,\alpha)$ is expansive and $X$ is connected.
Then   $(X, \alpha)$ is topologically rigid. 
\end{theorem}

As a  consequence, we get the following result.

\begin{corollary}
\label{c:endo-princ}
Let $\Gamma$ be a countable abelian group and let $f \in \Z[\Gamma]$.
Suppose that the principal algebraic dynamical system $(X_f,\alpha_f)$ associated with $f$ is expansive with $X_f$ connected and let $\tau \colon X_f \to X_f$ be a map.
Then the following conditions are equivalent:
\begin{enumerate}[\rm(a)]
\item
$\tau$ is an endomorphism of the dynamical system $(X_f,\alpha_f)$;
\item
 there exist $r \in \Z[\Gamma]$ and $t \in \Fix(X_f,\alpha_f)$ such that
$\tau(x) = r x + t$ for all $x \in X_f$.
\end{enumerate}
\end{corollary}

\begin{proof}
For each $r \in \Z[\Gamma]$,  the self-mapping of $X$ given by  $x \mapsto r x$ is 
$\Gamma$-equivariant since $\Z[\Gamma]$ is commutative.
Therefore,
the fact that (b) implies (a) follows from Proposition~\ref{p:charact-equiv-affine}.
\par
To prove the converse, suppose that $\tau$ is an endomorphism of the dynamical 
system $(X_f,\alpha_f)$. 
It follows from Theorem~\ref{t:bhat-rigid} that $\tau$ is affine. 
Therefore, by using Proposition~\ref{p:charact-equiv-affine}, there exist a continuous $\Z[\Gamma]$-module  morphism 
$\lambda \colon X_f \to X_f$ and $t \in \Fix(X_f,\alpha_f)$ such that
$\tau(x) = \lambda(x) + t$ for all $x \in X_f$.
As the ring $\Z[\Gamma]$ is commutative and 
$\widehat{X_f} = \Z[\Gamma]/ \Z[\Gamma] f$ is a cyclic $\Z[\Gamma]$-module,
there is  $r \in \Z[\Gamma]$ such that $\widehat{\lambda}(\chi) = r \chi$ for all $\chi \in \widehat{X_f}$.
Since $\lambda = \widehat{\widehat{\lambda}}$, it follows that $\lambda(x) = r x$
and hence $\tau(x) = r x + t$  for all $x \in X_f$.
\end{proof}

\section{Proof of the main result}
\label{sec:proof}

In this section, we present the proof of Theorem~\ref{t:main-result}.
So let $\Gamma$ be a countable abelian group  and $f \in \Z[\Gamma]$
such that $X_f$ is connected and $(X_f,\alpha_f)$ is expansive.
Also let $\tau$ be an endomorphism of $(X_f,\alpha_f)$, i.e.,
 a $\Gamma$-equivariant continuous map $\tau \colon X_f \to X_f$.
 We want to show that $\tau$ is surjective if and only if it is pre-injective.
 \par
By Corollary~\ref{c:endo-princ},
there exists $r \in \Z[\Gamma]$ such that
$\tau$ is affine with linear part $\lambda \colon X_f \to X_f$ given by
$\lambda(x) = r x$ for all $x \in X_f$.
Clearly $\tau$ is surjective if and only if $\lambda$ is.
As $X_f$ is compact,
we know that the surjectivity of $\lambda$ is equivalent to the injectivity of its Pontryagin dual
$\widehat{\lambda} \colon \widehat{X_f} \to \widehat{X_f}$.
Now we observe that $\widehat{\lambda}(\chi) = r \chi$ for all $\chi  \in \widehat{X_f}$.
As the $\Z[\Gamma]$-modules $\widehat{X_f}$ and $\Delta(X_f,\alpha_f)$ are isomorphic
by Corollary~\ref{c:homoclinic-group-dual},
the injectivity of $\widehat{\lambda}$ is equivalent to the injectivity of the endomorphism $\mu$ of
$\Delta(X_f,\alpha_f)$ defined by  $\mu(x) \coloneqq   r x$ for all $x \in \Delta(X_f,\alpha_f)$.
As $\mu$ is the restriction of $\lambda$ to
$\Delta(X_f,\alpha_f)$,
we conclude that the surjectivity  of $\tau$ is equivalent to the pre-injectivity of $\tau$ by 
using Proposition~\ref{p:pre-inj-affine}.

\section{Concluding remarks}

\subsection{Surjunctivity}
A dynamical system $(X,\alpha)$ is called \emph{surjunctive} if every injective endomorphism of 
$(X,\alpha)$ is surjective (and hence a homeomorphism).
As injectivity implies pre-injectivity,
we deduce from   Theorem~\ref{t:main-result}
that if $\Gamma$ is a countable abelian group and $f \in \Z[\Gamma]$ is such that  $(X_f,\alpha_f)$ is expansive
and $X_f$ is connected, then 
the dynamical system  $(X_f,\alpha_f)$ is surjunctive.
Actually, in the case $\Gamma = \Z^d$, this last result is a particular case 
of Theorem~1.5 in~\cite{bcsc-surjunctivity}
which asserts that if $\Gamma = \Z^d$ and $M$ is a finitely generated $\Z[\Gamma]$-module, then 
$(\widehat{M},\alpha_M)$ is surjunctive
(it is not required here that $\widehat{M}$ is connected nor that $(\widehat{M},\alpha_M)$ is expansive).

\subsection{Mixing}
Let $(X,\alpha)$ be an algebraic dynamical system with acting group $\Gamma$ and denote by $\mu$ the Haar probability measure on $X$.
One says that $(X,\alpha)$ is \emph{mixing} if
\begin{equation}
\label{e:def-mixing}
\lim\limits_{\gamma \to\infty}
\mu(B_{1} \cap \gamma B_{2})
=\mu(B_{1})\cdot\mu(B_{2})
\end{equation}
for all measurable subsets $B_1,B_2 \subset X$.
\par
Observe that if $A$ is a compact metrizable abelian group and $\Gamma$ is any infinite countable group, then the $\Gamma$-shift $(A^\Gamma,\sigma)$ is mixing
since~\eqref{e:def-mixing}  is trivially satisfied when 
$B_1$ and $B_2$ are cylinders.
\par
If $f \in \Z[\Gamma]$ is such that the system $(X_f,\alpha_f)$ is expansive (i.e., $f$ is inevertible in $\ell^1(\Gamma)$, then $(X_f,\alpha_f)$ is mixing
(cf.~\cite[Proposition~4.6]{lind-schmidt-survey-heisenberg}).
\par 
The examples below show that Theorem~\ref{t:main-result} becomes false if the hypothesis that the system 
$(X_f,\alpha_f)$ is expansive is replaced by the weaker hypothesis that it is mixing.

\begin{example}
Let $\T = \R/\Z$, $\Gamma = \Z^d$, $d \geq 1$,  
and consider the $\Gamma$-shift $(\T^\Gamma,\sigma)$
(this is  $(X_f,\alpha_f)$ for $f = 0 \in \Z[\Gamma]$).
Then the endomorphism $\tau$ of $(\T^\Gamma,\sigma)$ defined by
$\tau(x)(\gamma) = 2x(\gamma)$ for all $x \in \T^\Gamma$ and $\gamma \in \Gamma$,
is clearly surjective. 
However, $\tau$ is not pre-injective since the configuration $y \in \T^\Gamma$, defined by 
$y(\gamma) = 1/2 \mod 1$ if $\gamma = 0_\Gamma$ and $0$ otherwise, is a non-trivial element in the homoclinic group 
of $(\T^\Gamma,\sigma)$ and satisfies 
$\tau(y) = \tau(0) = 0$.
It follows that $(\T^\Gamma,\sigma)$ does not have the Moore property.
\end{example}

\begin{example}[cf.~]
Let $\Gamma = \Z$ and consider the  polynomial
\[
f = 1 -2u_1 + u_1^2 - 2u_1^3 + u_1^4 \in  \Z[u_1, u_1^{-1}] = \Z[\Gamma].
\]
The associated algebraic dynamical system $(X_f,\alpha_f)$ is conjugate 
to the system $(\T^4,\beta)$,
where $\beta$ is the action of $\Z$ on $\T^4$ generated by the companion matrix of $f$.
It is mixing since $f$ is not divisible by a cyclotomic polynomial
(cf.~\cite[Theorem~6.5.(2)]{schmidt-book}).
On the other hand,  $f$  has four distinct roots in $\C$,
two   on the unit circle, one inside and one outside.
As $f$ is irreducible over $\Q$, it follows that    
  the homoclinic group $\Delta(X_f,\alpha_f)$  is reduced to $0$ 
  (cf.~\cite[Example 3.4]{lind-schmidt}).
The trivial endomorphism of $(X_f,\alpha_f)$, that maps each $x \in X_f$ to $0$,
is pre-injective but not surjective.
Consequently, $(X_f,\alpha_f)$ does not have the Myhill property.
However,  $(X_f,\alpha_f)$ has the Moore property
since  each homoclinicity class of $(X_f,\alpha_f)$ is reduced to a single point,
so that every map with source set $X_f$  is pre-injective.
Note  that $(X_f,\alpha_f) = (\T^4,\beta)$ is topologically rigid since
every mixing toral automorphism is topologically rigid by
a result of Walters~\cite{walters}. 
\end{example}

\bibliographystyle{siam}

\begin{thebibliography}{10}

\bibitem{bartholdi-ejm-2010}
{\sc L.~Bartholdi}, {\em Gardens of {E}den and amenability on cellular
  automata}, J. Eur. Math. Soc. (JEMS), 12 (2010), pp.~241--248.

\bibitem{bartholdi:2016}
{\sc L.~Bartholdi and D.~Kielak}, {\em Amenability of groups is characterized
  by {M}yhill's theorem}, arXiv:1605.09133.

\bibitem{bhattacharya}
{\sc S.~Bhattacharya}, {\em Orbit equivalence and topological conjugacy of
  affine actions on compact abelian groups}, Monatsh. Math., 129 (2000),
  pp.~89--96.

\bibitem{bcsc-surjunctivity}
{\sc S.~Bhattacharya, T.~Ceccherini-Silberstein, and M.~Coornaert}, {\em
  Surjunctivity and topological rigidity of algebraic dynamical systems},
  arXiv:1702.06201, to appear in {E}rgodic {T}heory and {D}ynamical {S}ystems.

\bibitem{csc-ijm-2015}
{\sc T.~Ceccherini-Silberstein and M.~Coornaert}, {\em Expansive actions of
  countable amenable groups, homoclinic pairs, and the {M}yhill property},
  Illinois J. Math., 59 (2015), pp.~597--621.

\bibitem{csc-anosov-tori}
{\sc T.~Ceccherini-Silberstein and M.~Coornaert}, {\em A {g}arden of {E}den
  theorem for {A}nosov diffeomorphisms on tori}, Topology Appl., 212 (2016),
  pp.~49--56.

\bibitem{ceccherini}
{\sc T.~Ceccherini-Silberstein, A.~Mach{\`{\i}}, and F.~Scarabotti}, {\em
  Amenable groups and cellular automata}, Ann. Inst. Fourier (Grenoble), 49
  (1999), pp.~673--685.

\bibitem{deninger-schmidt}
{\sc C.~Deninger and K.~Schmidt}, {\em Expansive algebraic actions of discrete
  residually finite amenable groups and their entropy}, Ergodic Theory Dynam.
  Systems, 27 (2007), pp.~769--786.

\bibitem{gromov-esav}
{\sc M.~Gromov}, {\em Endomorphisms of symbolic algebraic varieties}, J. Eur.
  Math. Soc. (JEMS), 1 (1999), pp.~109--197.

\bibitem{lind-schmidt}
{\sc D.~Lind and K.~Schmidt}, {\em Homoclinic points of algebraic {${\bf
  Z}^d$}-actions}, J. Amer. Math. Soc., 12 (1999), pp.~953--980.

\bibitem{lind-schmidt-survey-heisenberg}
{\sc D.~Lind and K.~Shmidt}, {\em A survey of algebraic actions of the discrete
  {H}eisenberg group}, Uspekhi Mat. Nauk, 70 (2015), pp.~77--142.

\bibitem{machi-mignosi}
{\sc A.~Mach{\`{\i}} and F.~Mignosi}, {\em Garden of {E}den configurations for
  cellular automata on {C}ayley graphs of groups}, SIAM J. Discrete Math., 6
  (1993), pp.~44--56.

\bibitem{moore}
{\sc E.~F. Moore}, {\em Machine models of self-reproduction}, vol.~14 of Proc.
  Symp. Appl. Math., American Mathematical Society, Providence, 1963,
  pp.~17--34.

\bibitem{morris}
{\sc S.~A. Morris}, {\em Pontryagin duality and the structure of locally
  compact abelian groups}, Cambridge University Press, Cambridge-New
  York-Melbourne, 1977.
\newblock London Mathematical Society Lecture Note Series, No. 29.

\bibitem{myhill}
{\sc J.~Myhill}, {\em The converse of {M}oore's {G}arden-of-{E}den theorem},
  Proc. Amer. Math. Soc., 14 (1963), pp.~685--686.

\bibitem{schmidt-book}
{\sc K.~Schmidt}, {\em Dynamical systems of algebraic origin}, vol.~128 of
  Progress in Mathematics, Birkh\"auser Verlag, Basel, 1995.

\bibitem{walters}
{\sc P.~Walters}, {\em Topological conjugacy of affine transformations of
  tori}, Trans. Amer. Math. Soc., 131 (1968), pp.~40--50.

\end{thebibliography}

\end{document}